\documentclass[a4paper,12pt]{article}

\newenvironment{proof}{\noindent {\bf Proof:}}{\hfill $\Box$}

\newtheorem{theorem}{Theorem}

\newtheorem{corollary}{Corollary}

\newtheorem{remark}{Remark}

\newtheorem{example}{Example}
\usepackage{booktabs}

\usepackage{amsmath,amsfonts,amssymb,mathrsfs}
\usepackage[numbers]{natbib}
\usepackage{graphicx}
\usepackage{color}
\usepackage[normalem]{ulem}
\usepackage{bm}
\usepackage{caption}
\usepackage[T1]{fontenc}

\def\VDM{\mathrm{VDM}}
\def\x{\mathbf{x}}

\def\z{\mathbf{x}}
\def\p{\mathbf{p}}

\def\A{\mathbb{A}}

\def\y{\mathbf{y}}

\def\R{\mathbb{R}}

\def\N{\mathbb{N}}

\def\M{\mathbf{M}}

\def\C{\mathbb{C}}

\def\A{\mathbf{A}}

\def\B{\mathbf{B}}
\def\c{\mathbf{C}}

\def\z{\mathbf{z}}

\def\y{\mathbf{y}}
\def\d{\hat{d}}
\def\x{\mathbf{x}}

\def\c{\mathbf{c}}

\def\u{\mathbf{u}}
\def\v{\mathbf{v}}

\def\bvarepsilon{\boldsymbol{\varepsilon}}
\def\bxi{\boldsymbol{\xi}}
\def\balpha{\boldsymbol{\alpha}}
\def\bbeta{\boldsymbol{\beta}}
\def\bgamma{\boldsymbol{\gamma}}

\def\bpsi{\boldsymbol{\psi}}

\def\c{{\bf{c}}}

\def\balpha{\boldsymbol{\alpha}}
\def\bphi{\boldsymbol{\phi}}
\def\bpsi{\boldsymbol{\psi}}

\def\bmu{\boldsymbol{\mu}}

\def\1{\mathbf{1}}
\def\d{\mathrm{d}}

\begin{document}
	
	\title{\bf Approximate D-optimal design and\\ equilibrium measure\footnote{The second author  was supported by the AI Interdisciplinary Institute ANITI  funding through the french program
			``Investing for the Future PI3A" under the grant agreement number ANR-19-PI3A-0004. This research is also part of the programme DesCartes and is supported by the National Research Foundation, Prime Minister's Office, Singapore under its Campus for Research Excellence and Technological Enterprise (CREATE) programme.}}

	\author{Didier Henrion$^{1,2}$, Jean Bernard Lasserre$^{2,3}$}
	
	\footnotetext[1]{CNRS; LAAS; Universit\'e de Toulouse, France. }
	\footnotetext[2]{Faculty of Electrical Engineering, Czech Technical University in Prague, Czechia.}
	\footnotetext[3]{Toulouse School of Economics (TSE), Toulouse, France.}
	
	\date{Draft of \today}
	
	\maketitle
	
	\begin{abstract}
		We introduce a minor variant of 
		the approximate D-optimal design of experiments with a more general information matrix
		that takes into account the representation
		of the design space $S$. The main motivation (and result) is that if $S\subset\R^{d}$ is the unit ball, the unit box or the canonical simplex, then remarkably, for every dimension $d$ and every degree $n$,
		one obtains an optimal solution in closed form, namely the equilibrium measure of  $S$ (in pluripotential theory). Equivalently, for each degree $n$, the unique optimal solution is the vector of moments (up to degree $2n$) of the equilibrium measure of $S$.
		Hence finding an optimal design
		reduces to finding a cubature for the equilibrium measure,
		with atoms in $S$, positive weights, and exact up to degree $2n$. In addition, any resulting sequence of atomic D-optimal measures converges to the equilibrium measure of $S$ for the weak-star topology, as $n$ increases.
		Links with Fekete sets of points are also discussed. More general compact basic semi-algebraic sets are also considered, and a previously developed two-step design algorithm
		is easily adapted to this new variant of D-optimal design problem. 
	\end{abstract}
	
	\maketitle

\section{Introduction}

In this paper we consider the  approximate D-optimal design problem 
\begin{equation}
 \label{def-D-optimal}
 \max_{\phi\in\mathscr{P}(S)} \quad\log\det\, \M_n(\phi)\,,
\end{equation}
where $S\subset\R^d$ (the design space) is compact, $\mathscr{P}(S)$ is the set of probability measures on $S$, and
with $\v_n(\x):=(\x^{\balpha})$, $\balpha\in\N^d_n$, 
\[\M_n(\phi)\,:=\,\int_S \v_n(\x)\v_n(\x)^T\,d\phi(\x)\]
denotes the degree-$n$ moment matrix of $\phi$.

{The \emph{approximate} qualification 
is because one considers a (compact) set $S$ and a probability measure $\phi$ on $S$ rather than a discrete set 
of points and weights. %frequencies at which some of the points must be chosen. 
In their seminal paper on the equivalence theorem, Kiefer \& Wolfowitz \cite{kief-wolf} %consider the subset $\mathcal{C}\subset\mathscr{P}(S)$ of \emph{atomic} probability measures
%$\phi$ on $S$, supported on finitely many points\footnote{In fact, in view of Tchakaloff's theorem \cite{Tchaka} and its extension by Richter \cite{Richter,Schmudgen}, 
%one may indeed replace $\mathcal{C}$ with $\mathscr{P}(S)$ and solve \eqref{def-D-optimal}.}. 
show that $\phi^*$
%is an optimal solution of \eqref{def-D-optimal} (with $\mathcal{C}$ in lieu of $\mathscr{P}(S)$)
maximizes $\log\,\det \M_n(\phi)$ over $\mathscr{P}(S)$ if and only if %(i)
$\phi^*$ minimizes 
$\max_{\x\in S}K^\phi_n(\x,\x)$ (where $K^\phi_n(\x,\x)=\v_n(\x)\M_n(\phi)^{-1}\v_n(\x))$) over $\mathscr{P}(S)$, % and (ii) in addition,
attaining the value $s_n:={n+d\choose d}$. The polynomial $K^\phi_n(\x,\x)$ is a sum of squares and is called the degree-$2n$ Christoffel polynomial (the reciprocal of the Christoffel function) associated with $\phi$. Then 
\begin{equation}
\phi^*\,=\,\arg\max_{\phi\in \mathscr{P}(S)}\:\log\mathrm{det}\,\M_n(\phi)\,=\,\arg\min_{\phi\in\mathscr{P}(S)}\,\max_{\x\in S}K^{\phi}_n(\x,\x)
\end{equation}
 and $\max_{\x\in S}\,K^{\phi^*}_n(\x,\x)\,=\,s_n$ is attained  at all points of the support of $\phi^*$.
%Moreover, if the support of $\phi^*$ consists of exactly $s_n$ distinct points, then they all have same weight $1/s_n$ 
%and they are called \emph{Fekete} points of $S$. However,
%this latter property happens in rather exceptional cases;
%see e.g. \cite{Bloom,Bos-1,Bos-2}.
}

In this paper, we restrict ourselves to three specific sets $S$,  namely the Euclidean unit ball, the unit box, and the canonical simplex. We introduce the following slight variant of \eqref{def-D-optimal}:
\begin{equation}
\label{D-variant}
\max_{\phi\in\mathscr{P}(S)} \quad\log\,\mathrm{det}\,\M_n(\phi)+ \sum_{g\in G}\log\,\mathrm{det}\,\M_{n-d_g}(g\phi) \,\end{equation}
where 
$g\phi$ is the measure {on $S$} which is absolutely continuous with respect to $\phi$ with density $g$, $d_g=\lceil\mathrm{deg}(g)/2\rceil$, and {$G\subset\R[\x]$ is  an appropriate set of generators of $S$,  i.e. $S=\{\x: g(\x)\geq0\,\:\forall g\in G\}$.}
For instance if $S$ is the Euclidean unit ball then 
$G$ is the singleton $\x\mapsto g(\x):=1-\Vert\x\Vert^2$, and \eqref{D-variant} reads:
\[\max_{\phi\in\mathscr{P}(S)} \quad\log\,\mathrm{det}\,\M_n(\phi)+ \log\,\mathrm{det}\,\M_{n-1}(g\phi)\,.\]
That is, one has replaced the information matrix $\M_n(\phi)$ with the new information matrix
\[\left[\begin{array}{cc}\M_n(\phi) &0\\0&\M_{n-1}(g\phi)\end{array}\right]\,.\]
Notice that the boundary $\partial S$ of $S$ is now clearly involved via the polynomial $g$ in the measure $g\phi$. 
%In the univariate case $S=[-1,1]$ and $x\mapsto g(x):=1-x^2$, our framework shares some similarities with the generalization of D-optimal design proposed in \cite{dette,kiefer2,lauter}
%. In \cite{dette,lauter}, with a so-called \emph{efficiency} function of the form $\lambda(x):=(1+x)^u(1-x)^v$ with $u,v\in\{0,1\}$, $\beta_{n-1}=n/(2n+1)$, $\beta_n=n+1/(2n+1)$, 
%one minimizes the criterion
%\[\frac{\beta_{n-1}\log\,\mathrm{det}\,\M_{n-1}(\lambda\phi)}{n}+
%\frac{\beta_n\log\,\mathrm{det}\,\M_{n}(\lambda\phi)}{n+1}\,,\]
%which differs from \eqref{D-variant}.
%%, where the measure $g\phi$ is involvedin $\M_{n-1}$ and not in $\M_n$. 
%That is, in \cite{dette,lauter} 
%where the authors consider the 
%product of information matrices related to $n+1$ different regressions models.
% and with the \emph{same} efficiency function $\lambda$, whereas in \eqref{D-variant} one also considers several 
%regression models but each with a specific efficiency function. 
In \eqref{D-variant} the matrix $\M_{n-1}(g\phi)$ is associated with a regression model of the form
\[y\,=\,v_n(\x)^T\theta+\frac{\varepsilon}{\sqrt{g(\x)}}\,,\]
i.e., when the noise explodes close to the boundary of $S$\footnote{The authors thank Y. de Castro and F. Gamboa for providing such a simple interpretation.}.

\paragraph{Contribution}
We show that {for certain specific domains $S$}, for every degree $n$ and every dimension $d$, the \emph{equilibrium measure} 
$\phi^*$ of $S$ (in pluripotential theory \cite{klimek}) is an optimal solution of the variant \eqref{D-variant}
of the D-optimal design problem \eqref{def-D-optimal}.
This variant \eqref{D-variant} takes into account explicitly the boundary of $S$ and its criterion tends to favor points in the interior of $S$ (as points in $\partial S$  do not contribute
to some of the  information matrices $\M_{n-d_g}(g\phi)$). 
The resulting optimality conditions link the variant \eqref{D-variant} with a generalized (polynomial) Pell's equation investigated in \cite{lass-cras,lass-tams}. 

{
Notice that as $S$ is compact, by Richter-Tchakaloff's theorem \cite[Theorem 1.24]{Schmudgen},
problem \eqref{def-D-optimal} has always an optimal measure supported on finitely many points.} 
Hence:
\begin{center}
\emph{The variant \eqref{D-variant} of the D-optimal design problem \eqref{def-D-optimal} reduces to
finding a cubature for the equilibrium measure $\phi^*$ of $S$, with positive weights, atoms in $S$, and exact up to degree $2n$.}
\end{center}
{Appropriate techniques from numerical analysis to compute cubatures \cite{cools-1,cools-2,Gautschi} could then be combined with the arsenal of existing techniques. Alternatively, the two-step design algorithm provided in \cite{dghl19} to solve
\eqref{def-D-optimal}, reduces only to step-2 as the output of step $1$ (moments of the atomic measure supported on an optimal design) is now readily available.}

In addition, if $\nu^*_n$ is any atomic  probability measure, optimal solution of \eqref{D-variant} (hence whose support is a degree-$n$ optimal design), then  the resulting sequence $(\nu^*_n)_{n\in\N}$ 
converges weak-star to the equilibrium measure $\phi^*$ of $S$, as does
any sequence of probability measures equi-supported on so-called
\emph{Fekete} points of $S$. However, as already noted in the literature linking approximation theory with the D-optimal design problem, probability measures supported on Fekete points are rarely optimal solutions to the standard D-optimal design problem \eqref{def-D-optimal}; see e.g. \cite{Bloom,Bos-1,Bos-2}.

Last but not least, for the univariate case $S=[-1,1]$, the atomic probability measure $\nu^*_n$ supported on the $n+1$ roots  of the degree-$(n+1)$ Chebyshev polynomial of first kind (and with equal weights) is optimal for \eqref{D-variant}.
Hence the support of the celebrated Gauss-Chebyshev quadrature is an optimal solution of the variant \eqref{D-variant} of \eqref{def-D-optimal}. 
We emphasize that  $\nu^*_n$ \emph{cannot} be an optimal solution of the standard D-optimal design problem \eqref{def-D-optimal}. A  tensorized version of this result holds for the unit box $[-1,1]^d$. 
Interestingly, a design on Chebyshev points of $[-1,1]$ has been shown to be $\emph{\c}$-optimal\footnote{Given a vector $\c$, a design is $\c$-optimal if it minimizes $\phi\mapsto \sup_\v\c^T\v/\v^T\M_n(\phi)\v$ see e.g. \cite{studden-1}.} in \cite{hoel}; see also
\cite{studden-1} for optimal design on Chebyshev points.

Finally, we extend the variant \eqref{D-variant} to the case of arbitrary compact basic semi-algebraic sets $S\subset\R^d$. Of course for such general sets $S$, 
an optimal solution $\phi^*\in\mathscr{P}(S)$
of \eqref{D-variant} is not available in closed form any more. However, links with the equilibrium measure are still available asymptotically as the degree $n$ increases. Moreover the two-step design algorithm provided in \cite{dghl19} to solve
\eqref{def-D-optimal} can be easily adapted. In the convex relaxation defined in step 1
of the algorithm, it suffices to replace 
the log-det criterion of \eqref{def-D-optimal} with that of \eqref{D-variant}.
Then step 2 of the algorithm remains exactly the same. 

So in summary, for the three special geometries (unit ball,  unit box, and simplex), and for every dimension $d$ and every degree $n$, the equilibrium measure of $S$ is 
an optimal solution of the variant \eqref{D-variant} of D-optimal design. 
%Hence computing an optimal design reduces to computing a cubature for the equilibrium measure, with positive weights, support on $S$, and exact up to degree $2n$.
Moreover, for more general semi-algebraic sets $S$,
this variant also provides (asymptotically) connections with the equilibrium measure of $S$. It still remains to investigate how this variant compares with the classical version \eqref{def-D-optimal} from a statistical point of view.

\section{Notation, definitions and preliminaries}
\subsection{Notation and definitions}

Let $\R[\x]_n \subset \R[\x]$ denote the space of polynomials 
in the variables $\x=(x_1,\ldots,x_d)$, of total degree at most $n$. Let $\Sigma[\x]\subset\R[\x]$ be the set of sums of squares (SOS) polynomials and 
let $\Sigma[\x]_n\subset\R[\x]_{2n}$
be its subset of SOS polynomials of total degree 
at most $2n$. We also denote by $\1\in\R[\x]$, the constant polynomial with value $1$.
Let $\N$ be the set of natural numbers and 
$\N^d_n:=\{\balpha\in\N^d: \vert\balpha\vert\:(=\sum_i\alpha_i )\,
\leq n\}$, and $s_n:={d+n\choose d}$.
A polynomial $p\in\R[\x]_n$ reads
\[\x\mapsto p(\x)\,=\,\sum_{\balpha\in\N^d_n}p_{\balpha}\,\x^{\balpha}\,,\]
where $\p=(p_{\balpha})\in\R^{s_n}$ is the vector of coefficients of $p$ in the monomial basis $\v_n(\x):=(\x^{\balpha})_{\balpha \in \N^d_n}$ with $\x^{\balpha}:=\prod_{i=1}^d x^{\alpha_i}_i$.
With $S\subset\R^d$ compact, denote by $\mathscr{M}(S)_+$ the convex cone of Borel (positive) 
measures on $S$ and $\mathscr{P}(S)\subset\mathscr{M}(S)_+$
its subset of probability measures on $S$. Denote also by $\mathscr{C}(S)$ the space of continuous functions on $S$.
%With a given real sequence $\bphi=(\phi_{\balpha})_{\balpha\in\N^d}$ is associated the 
%Riesz linear functional $\phi\in\R[\x]^*$ defined by
%\[p\mapsto \phi(p)\,:=\,\sum_{\balpha\in\N^d} p_{\balpha}\,\phi_{\balpha}\,,\quad\forall p\in\R[\x]\,.\]

Given  a set of $s_n$ points $\{\x_1,\ldots,\x_{s_n}\}\subset S$
denote by $\VDM(\x_1,\ldots,\x_{s_n})\in\R^{s_n\times s_n}$ the Vandermonde matrix associated with $\{\x_1,\ldots,\x_{s_n}\}$ {and the monomial basis}. %Let $\R^d_{++}:=\{\,\x\in\R^d: \x>0\,\}$.

\paragraph{Moment and localizing matrix}
{With a given real sequence $\bphi=(\phi_{\balpha})_{\balpha\in\N^d}$ (in bold), and $n\in\N$,
is associated the Riesz linear functional $\phi\in\R[\x]_n^*$ (not in bold) defined by:
\[p\:(=\sum_{\balpha} p_{\balpha}\,\x^{\balpha})\quad\mapsto\quad \phi(p)\,=\,\sum_{\balpha\in\N^d_n}p_{\balpha}\,
\phi_{\balpha}\,,\quad\forall p\in\R[\x]_n\,.\]}
The moment matrix $\M_n(\bphi)$ (or $\M_n(\phi)$) associated with $\bphi$
is the real symmetric matrix with rows and columns indexed by $\N^d_n$, and with entries
\[\M_n(\bphi)(\balpha,\bbeta)\,:=\,\phi_{\balpha+\bbeta}\,,\quad\balpha,\bbeta\in\N^d_n\,.\]
If $\bphi$ has a representing measure $\phi$ {(i.e., 
$\phi_{\balpha}=\int \x^{\balpha}d\phi$ for all $\balpha\in\N^d$)} then 
$\M_n(\bphi)\succeq0$ for every $n$. Conversely, $\M_n(\bphi)\succeq0$ for every $n$, is a necessary but not sufficient condition for $\bphi$ to
have a representing measure on $\R^d$.

Given a polynomial $g\in\R[\x]$, $\x\mapsto g(\x):=\sum_{\bbeta} g_{\bbeta}\,\x^{\bbeta}$, and a real sequence $\bphi=(\phi_{\balpha})_{\balpha\in\N^d}$, the localizing matrix 
$\M_n(g\,\bphi)$ associated with $g$ and $\bphi$
is the moment matrix associated with the Riesz linear functional
$g\phi\in\R[\x]^*_n$ defined by:
\[p\mapsto (g\,\phi)(p)\,=\,\phi(g\,p)\,,\:\forall p\in\R[\x]_n\,;\quad (g\bphi)_{\balpha}\,:=\,\sum_{\bbeta\in\N^d}g_{\bbeta}\,\phi_{\bbeta+\balpha}\,,\quad\forall \balpha\in\N^d_n\,.\]
If $g=\1$, the localizing matrix $\M_n(g\,\bphi)$  is simply the moment matrix associated with $\bphi$, and if $\bphi$ has a representing measure supported on the set $\{\,\x:\:g(\x)\geq0\,\}$, then $\M_n(g\,\bphi)\succeq0$ for all $n$.

%{\paragraph{Preordering}

%Given polynomials $g_1,\ldots,g_m\in\R[\x]$, the set of polynomials
%\begin{equation}
%\label{preorder}
%T(g_1,\ldots,g_m)\,:=\,\{\:\x\mapsto\sum_{\bvarepsilon\in\{0,1\}^m}\sigma_{\bvarepsilon}(\x)\,\prod_{j=1}^mg_j(\x)^{\varepsilon_j}\,\:\sigma_{\bvarepsilon}\in\Sigma[\x]\:\}\,.
%\end{equation}
%is called a \emph{preordering}. If $S=\{\,\x: g_j(\x)\geq0\,,\:j=1,\ldots,m\,\}$
%then $p\geq0$ on $S$ for every $p\in T(g_1,\ldots,g_m)$.
%Conversely, if $S$ is compact then Schm\"udgen's Positivstellensatz \cite[Theorem 2.13]{lass-book} states that $[\,p\in\R[\x]$ and $p>0$ on $S$ ] $\Rightarrow p\in T(g_1,\ldots,g_m)$.

%In addition, let $g_{\bvarepsilon}:=\prod_{j=1}^m g_j^{\varepsilon_j}$ for every $\bvarepsilon\in\{0,1\}^m$. Then \cite[Theorem 2.44(a)]{lass-book} (the real analysis facet of Schm\"udgen's theorem) states that a real sequence $\bphi=(\phi_{\balpha})_{\balpha\in\N^d}$  has a representing measure on $S$
%if and only if 
%$\M_n(g_{\bvarepsilon}\,\bphi)\succeq0$ for all $\bvarepsilon\in\{0,1\}^m$ and all $n\in\N$.}

\paragraph{Christoffel-Darboux kernel and Christoffel function}
 Given a compact set $S\subset\R^d$ and a Borel measure $\phi$ on $S$ such that $\M_n(\phi)\succ0$ for all $n$, let $(P_{\balpha})_{\balpha\in\N^d}\subset\R[\x]$ be a family of polynomials that are orthonormal w.r.t. $\phi$, i.e.:
 \[\int_S P_{\balpha} P_{\bbeta}\,d\phi\,=\,\delta_{i=j}\,,\quad \forall\balpha,\bbeta\in\N^d\,,\]
 which is guaranteed to exist. Then for every $n$, the kernel
 \[(\x,\y)\mapsto K^\phi_n(\x,\y)\,:=\,\sum_{\balpha\in\N^d_n}P_{\balpha}(\x)\,P_{\balpha}(\y)\,,\quad\forall \x,\y\in\R^d\,,\]
  is called the Christoffel-Darboux (CD) kernel\,, the polynomial $\x \mapsto K^\phi_n(\x,\x)/s_n$ is called the (normalized) CD polynomial,
 and the rational function
 \[\x\mapsto \Lambda^\phi_n(\x)\,=\,1/K^\phi_n(\x,\x)\,,\quad\forall \x\in\R^d\,,\]
 is called the Christoffel function. Alternatively
 \begin{eqnarray}
\label{eq:ch1}
 \Lambda^\phi_n(\x)&=&(\v_n(\x)^T\M_n(\phi)^{-1}\v_n(\x))^{-1}\\
\label{eq:ch2}
&=&
 \min_{p\in\R[\x]_n}\{\,\int_S p^2\,d\phi:\: p(\x)\,=\,1\,\}\,,\quad\forall \x\in\R^d\,.\end{eqnarray}
 {Relation \eqref{eq:ch1} is called the ABC theorem in \cite{Simon2008}
 and for \eqref{eq:ch2} see e.g. \cite[Theorem 3.1]{acm}.}
 
 \paragraph{Equilibrium measure}
 The notion of equilibrium measure associated with a given set originates from logarithmic potential theory 
(working with a compact set $E\subset\mathbb{C}$ in the univariate case). It minimizes the energy functional
\begin{equation}
\label{logarithmic}
I(\phi)\,:=\,\int\int \log{\frac{1}{\vert z-t\vert}}\,\
d\phi(z)\,d\phi(t)\,,
\end{equation}
over all Borel probability measures $\phi$ supported on $E$. For instance if $E$ is the interval $[-1,1]\subset\C$ then the arcsine (or Chebyshev) distribution 
$\mu=\d x/\pi\sqrt{1-x^2}$ is an optimal solution. 
Generalizations have been obtained in 
the multivariate case via pluripotential theory in $\mathbb{C}^d$. In particular, if $E\subset\R^d\subset\mathbb{C}^d$ 
is compact then its equilibrium measure (let us denote it by $\mu$) is equivalent to the Lebesgue measure 
on compact subsets of $\mathrm{int}(E)$. It has an even explicit expression if $E$ is convex and symmetric 
about the origin; see e.g. \cite[Theorems 1.1 and 1.2]{bedford}. Several examples of sets $E$ 
with its equilibrium measure given in explicit form can be found in \cite{Baran}. 
Importantly, the appropriate approach to define the (intrinsic) equilibrium measure $\mu$ of a compact subset of $\R^d$ with $d>1$, is to consider $\R^d$ as a subset of $\C^d$ and invoke pluripotential theory with its tools from complex analysis (in particular, plurisubharmonic functions ({and their regularizations}) and the Monge-Amp\`ere operator). 
For more details on equilibrium measures and pluripotential theory, the interested reader is referred to \cite{Baran,bedford,klimek}, the discussion in \cite[Section 6.8, p. 297]{Kirsch} as well as \cite[Appendix B]{Saff-Totik}, \cite{Levenberg-survey}, and the references therein.  In the sequel, when we speak about the equilibrium measure of  a compact subset $E\subset\R^d$, we refer to that in pluripotential theory (i.e., with $E$ considered as a subset of $\C^d$).

\subsection{Background on approximate D-optimal design}
The approximate D-optimal design problem is well-known and originates in  statistics.  Let $S\subset\R^d$ be a  compact set with nonempty interior, and let
$n\in\N$ be fixed. For a design $\bxi=(\bxi_1,\ldots,\bxi_r)$ with positive weights $\bgamma=(\gamma_i)_{i\leq r}$ ($\gamma_i$ is the frequency at which
$\bxi_i$ is chosen)
the matrix $\M_n(\bxi)=\sum_{i=1}^r \gamma_i\,\v_n(\bxi_i)\v_n(\bxi_i)^T$ 
is called 
the information matrix of the design $\bxi$. {Among several statistical criteria  in parameter estimation, maximizing 
(over all such atomic designs) the logarithm of the determinant
of the information matrix is a popular one. Optimizing the same criterion over \emph{all} probability measures on $S$
yields problem
\eqref{def-D-optimal}, and as one optimizes over $\mathscr{P}(S)$, an optimal solution is called an \emph{approximate} D-optimal design. For a general overview on optimal experimental design the interested reader is referred to the seminal papers \cite{kief-wolf} and the recent tutorial \cite{acta}.}

If $S$ has nonempty interior then it is shown in \cite[Theorem 1]{dghl21} that \eqref{def-D-optimal} has  an optimal (not necessarily unique) atomic measure supported on 
$m$ points $\x_j\in S$, $j=1,\ldots,m$, where 
$s_n\leq m\leq s_{2n}$. That is, there exists a weight vector $0<\bgamma\in\R^m$ such that
\[\phi^*\,=\,\sum_{j=1}^m \gamma_j\,\delta_{\x_j}\,,\quad \sum_{j=1}^m \gamma_j=1\,.\]
Next, with $K^{\phi^*}_n(\x,\y)$ being the CD kernel 
associated with $\phi^*$,
\begin{eqnarray}
\label{a1}
s_n-K^{\phi^*}_n(\x,\x)&\geq&0\,,\quad\forall \x\in\,S\\
\label{a2}
s_n-K^{\phi^*}_n(\x_j,\x_j)&=&0\,,\quad\forall j=1,\ldots,m\,.
\end{eqnarray}
See e.g. the equivalence theorem \cite[Theorem, p. 364]{kief-wolf} and \cite{Bloom}. 

While an optimal atomic measure is not necessarily unique, the resulting optimal
moment matrices $\M_n$ are all identical (by strict concavity of the criterion).
It turns out that from \eqref{a1}-\eqref{a2} 
%the characterization of an optimal atomic solution $\phi^*$ 
one  may also link \eqref{def-D-optimal} 
with approximation theory, orthogonal polynomials, Fekete points and Fej\'er points.

A set of $s_n$ points $\{\x^*_1,\ldots,\x^*_{s_n}\}\subset S$ is  a Fekete set if it maximizes 
$\mathrm{det}(\VDM(\x_1,\ldots,\x_{s_n}))$ 
(the Vandermonde determinant)
among all sets of $s_n$ points of $S$.  Observe that  if $\mu$
is  the $s_n$-atomic measure with equal weights  
 $\frac{1}{s_n}\sum_{i=1}^{s_n}\delta_{\x_i}$,
 then one obtains 
{\[\M_n(\mu)\,=\,\frac{1}{s_n}\VDM(\x_1,\ldots\x_{s_n})\VDM(\x_1,\ldots\x_{s_n})^T\,,\]}
and {$\log\,\mathrm{det}(\M_n(\mu))=2\log\,\mathrm{det}(\VDM(\x_1,\ldots\x_{s_n}))-s_n\log s_n$.}

Fekete points are in turn connected with the so-called \emph{equilibrium measure}
in pluri-potential theory \cite{klimek}. For instance 
the sequence of discrete probability measures 
$(\nu_n)_{n\in\N}$ equi-supported on $s_n$ Fekete points converges 
to the equilibrium measure of $S$ for the weak-star topology on the space
of signed measures on $S$.  See e.g. \cite{berman}, \cite[Theorem 4.5.1]{book}, and the references therein.

A set of $s_n$ points $\{\x_1,\ldots,\x_n\}\subset S$ is  a Fej\'er  set
if  {$\max_{\x\in S}\sum_{j=1}^N |\ell_j(\x)|^2=1$}, where $\ell_1,\ldots,\ell_{s_n}$ 
are  the Lagrange interpolation polynomials {associated with the points $\x_1,\ldots\x_{s_n}$ and the space $\R[\x]_n$}. It turns out that a set of Fej\'er points is also a  set of Fekete points, {whereas the reverse is not true in general \cite{Bos-1}.}  Hence it is natural to ask when a set of Fekete points is
also the support of an atomic measure $\phi^*$ 
in \eqref{a1}-\eqref{a2}, with equal weights.
In fact, an $s_n$-atomic
measure is D-optimal in \eqref{def-D-optimal} if and only if it is equally weighted and 
its support is a Fej\'er set. See for instance \cite{Bos-1,Bos-2}. However, and as noted in \cite{Bos-2}, such a situation is rather exceptional and not to be expected.

{
\subsection{Convex relaxations}
\label{algo-description}
As stated, problem \eqref{def-D-optimal} is intractable because one does not know 
how to optimize over $\mathscr{P}(S)$ for arbitrary sets $S\subset\R^d$. 
A typical numerical approach is to discretize the design space $S$.
%({\color{blue} quid des m\'ethodes existantes?})
However, if $S$ is the compact basic semi-algebraic set 
$\{\,\x\in\R^d: g_j(\x)\geq0\,,\: j=1,\ldots,m\,\}$ with $g_j\in\R[\x]$ for all $j$,
then the authors in \cite{dghl19} have proposed a two-step numerical procedure with
good results in a number of cases. In particular, this numerical procedure 
is mesh-free as it avoids discretizations of the design set $S$.

Letting $g_0=\1$, associated with \eqref{def-D-optimal}
is the following convex optimization problem:
\begin{equation}
\label{mom-conv-relax}
\max_{\bphi\in\R^{s_{2n}}}\,\{\,\log\mathrm{det}\,\M_n(\bphi)\,:\: \phi(\1)=1\,;\: 
\M_{n-r_j}(g_j\,\bphi)\,\succeq\,0\,,\:j=0,\ldots,m\,\}\,,
\end{equation}
where the maximization is over real sequences of pseudo-moments $\bphi=(\phi_{\balpha})_{\balpha\in\N^d_{2n}}$, that is, sequences which do not necessarily have a representing probability measure $\phi$ on $S$. For this reason \eqref{mom-conv-relax} is a \emph{relaxation} of \eqref{def-D-optimal}. {Indeed the probability measures that solve \eqref{def-D-optimal} form an equivalence class of measures with identical moments up to order $2n$, and so in a first step one may instead optimize 
over the set of moments up to order $2n$. As this set does not have a tractable characterization, one next
relaxes the problem to \eqref{mom-conv-relax} over pseudo-moments $\bphi\in\R^{s_{2n}}$ that satisfy
necessary conditions to be true moments.}

Then in the two-step algorithm proposed in \cite{dghl19} to solve \eqref{def-D-optimal}: 

-- Step-1 solves the relaxation \eqref{mom-conv-relax}, a convex optimization problem which can be solved efficiently (at least for reasonable dimensions), and one expects to obtain an optimal sequence $\bphi^*=(\phi^*_{\balpha})_{\balpha\in\N^d_{2n}}$ having a representing measure on $S$. That is,
one expects to obtain a sequence $\bphi^*$ of true moments instead of pseudo-moments.

--  Step-2 of the procedure consists of finding a {\sl flat extension} $\bpsi$ of the vector $(\phi^*_{\balpha})_{\balpha\in\N^d_{2n}}$, which is a vector of pseudo-moments $(\psi_{\balpha})_{\balpha\in\N^d_{2t}}$ with $t>n$ such that $\psi_{\balpha}=\phi^*_{\balpha}$ for all $\balpha \in \N^d_{2n}$ and $\text{rank}\:\M_t(\psi)=\text{rank}\:\M_{t-s}(\psi)=:r$ for some positive integer $s$. Then from this flat extension, one can use the algorithm of \cite[Section 6.1.2]{lass-book} to extract $r$ points of $S$, which provide a D-optimal design (the support of a D-optimal atomic measure $\nu_n\in\mathscr{P}(S)$). If $\bphi^*$ has indeed a representing measure on $S$,
then a flat extension can be obtained by solving a hierarchy of semidefinite programs for increasing values of $t=n+1,n+2,\ldots$, by minimizing (with respect to the vector of pseudo-moments $(\psi_{\balpha})_{\balpha\in\N^d_{2t}}$) the linear functional
$\psi(p)$ for a random strictly positive polynomial $p$, 
subject to the constraints $\M_{t-r_j}(\psi) \succeq 0$, $\forall j=0,\ldots,m$, until the condition $\text{rank}\:\M_t(\psi)=\text{rank}\:\M_{t-s}(\psi)$ is satisfied for some positive integer $s$. 

For more details and examples, the interested reader is referred to \cite[Section 5]{dghl21}.}

\section{Main result}
\label{geometry}
As we next see, the main result of this paper is obtained from \cite{lass-cras,lass-tams} 
and it is detailed for each of the three cases considered (unit ball, unit box, canonical simplex).
What we emphasize here is the significance of such results for the D-optimal design problem. Namely, by using results from \cite{lass-tams}, one shows that our proposed variant of the D-optimality criterion for optimal design, 
reveals quite strong links with the equilibrium measure $\phi^*$ of $S$.
In this variant appears a term which  involves the boundary $\partial S$ of the design space $S$. As a result, 
for three important sets $S\subset\R^d$, and for all dimensions $d$ and all degrees $n$,  
$\phi^*$ is an optimal solution of \eqref{D-variant}. Therefore an optimal design can be obtained from any cubature 
associated with $\phi^*$, provided that it has positive weights, atoms in $S$, and is exact up to degree $2n$.
{Define :

-- the Euclidean unit ball $S^{O}:=\{\,\x\in\R^d:\:1-\Vert\x\Vert^2\geq0\,\}$, with associated equilibrium measure
{
\[\phi^{O}\,=\,\frac{d\x}{\pi^d\,\sqrt{1-\Vert\x\Vert^2}}\,,\]}

-- the unit box $S^{\square}:=\{\,\x\in\R^d:\:1-x_j^2\geq0\,,\: j=1,\ldots,d\,\}$, 
with associated equilibrium measure
{
	\[\phi^{\square}\,=\,\frac{d\x}{\pi^d\,\sqrt{(1-x_1^2)\cdots (1-x_d^2)}}\,,\]}

-- and the canonical simplex
$S^{\triangle}:=\{\,\x\in\R^d_+:\:1-\sum_{j=1}^d x_j\geq0\,\}$
with associated equilibrium measure
{
	\[\phi^{\triangle}\,=\,\frac{d\x}{\pi^d\,\sqrt{x_1\cdots x_d\cdot(1-\sum_{j=1}^d x_j)}}\,.\]}
For every integer $n$, define following subsets of  $\R[\x]$:
\begin{enumerate}
\item $G_n^{O}:=\{\,\x\mapsto 1\,;\quad \x\mapsto 1-\Vert\x\Vert^2\,\}$.
\item $G_n^{\square}:=\{\,\x\mapsto \displaystyle\prod_{j=1}^d(1-x_j^2)^{\varepsilon_j}\,:\: \bvarepsilon\in\{0,1\}^d\,;\: \vert\bvarepsilon\vert\leq n\,\}$.
\item $G_n^{\triangle}:=\{\,\x\mapsto x_1^{\varepsilon_1}\cdots x_d^{\varepsilon_d}(1-\sum_{j=1}^d  x_i)^{\varepsilon_{d+1}}\,:\: \bvarepsilon\in\{0,1\}^{d+1}\,;\:\vert\bvarepsilon\vert\in 2\N \,;\:\vert\bvarepsilon\vert\leq n/2\,\}$,
\end{enumerate}
associated respectively with $S^{O}$, $S^{\square}$, and $S^{\triangle}$. Of course one also has:
\[S^{O}\,=\,\{\,\x: g(\x)\geq0\,,\:\forall g\in G_n^{O}\,\}\,;\quad
S^{\square}\,=\,\{\,\x: g(\x)\geq0\,,\:\forall g\in G_n^{\square}\,\}\,;\quad
S^{\triangle}\,=\,\{\,\x: g(\x)\geq0\,,\:\forall g\in G_n^{\triangle}\,\}\,.\]
{Notice that the sets $G_n^{\square}$ and $G_n^{\triangle}$ do not form \emph{minimal} sets of
generators for the sets $S^{\square}$ and $S^{\triangle}$, respectively. With minimal sets, we do not know whether our results are still valid.}

Next, fix $n\in\N$, $S^{\star}:=S^{O}, S^{\square}$, or $S^{\triangle}$, and 
instead of \eqref{def-D-optimal}, consider now the related optimization problem:
\begin{equation}
 \label{def-new}
 \rho_n=\max_{\phi\in\mathscr{P}(S^{\star})} \quad\sum_{g\in G^{\star}_n}\log\det\, \M_{n-d_g}(g\,\phi)
 \end{equation}
or equivalently,
\begin{equation}
 \label{def-new-min}
 -\rho_n=\min_{\phi\in\mathscr{P}(S^{\star})} \quad\sum_{g\in G^{\star}_n}\log\det\, \M_{n-d_g}(g\,\phi)^{-1}\,,
\end{equation}
where for every $g\in\R[\x]$, $d_g:=\lceil\mathrm{deg}(g)/2\rceil$. 

When comparing \eqref{def-new} with \eqref{def-D-optimal} observe that 
we have simply replaced the information matrix $\M_n(\phi)$ with the new 
block-diagonal information matrix with $\vert G^{\star}_n\vert$ blocks, and
where each diagonal block is the localizing matrix
$\M_{n-d_g}(g\,\bphi)$, $g\in G^{\star}_n$. For instance if $\star=O$ (i.e., $S^{\star}(=S^{O})$ is the Euclidean unit ball), the localizing matrix has two diagonal blocks and reads:
\[\left[\begin{array}{cc}\M_n(\bphi)&0\\
0&\M_{n-1}(g\,\bphi)\end{array}\right]\quad\mbox{with $g(\x)=1-\Vert\x\Vert^2$,}\,\quad n\,\in\,\N\,.\]
%associated with the measures $\phi$ and $g\phi$. 
The latter information matrix takes into account the 
distance to the boundary 
$\partial S^{O}=\{\x : \Vert\x\Vert=1\}$ of the design space $S^{O}$. Indeed, for a design $\bxi=(\bxi_1,\ldots,\bxi_r)$
with weights $\bgamma=(\gamma_i)_{i\leq r}$, the information matrix $\M_{n-1}(g\,\nu_{\bxi})$
(with $\nu_{\bxi}=\sum_i\gamma_i\delta_{\bxi_i}$) reads
\[\M_{n-1}(g\,\nu_{\bxi})\,=\,\sum_{i=1}^r \gamma_i\,(1-\Vert\bxi_i\Vert^2)\,\v_{n-1}(\bxi_i)\v_{n-1}(\bxi_i)^T\,,\]
and so points $\bxi_i\in\partial S^{O}$ do \emph{not} contribute to $\M_{n-1}(g\nu_{\bxi})$. So the criterion
in \eqref{def-new} put more weight on points in $\mathrm{int}(S^{\star})$ than on points in $\partial S^{\star}$.

Next, the analogue for \eqref{def-new}  of the convex relaxation \eqref{mom-conv-relax} for \eqref{def-D-optimal}
reads:
\begin{equation}\label{def-new-relax}
\tau_n=\max_{\bphi}\,{\left\{ \,\sum_{g\in G^\star_n} \log\mathrm{det}\,\M_{n-d_g}(g\,\bphi):\:\phi(\1)=1\,;\:\M_{n-d_g}(g\,\bphi)\succeq0\,,\: \forall g\in G^\star_n\,\right\}}\,,
\end{equation}
where the maximization is over vectors
$\bphi=(\phi_{\balpha})_{\balpha\in\N^d_{2n}}$ of pseudo-moments.
We also consider  the convex optimization problem :
\begin{equation}
 \label{def-new-dual}
% \begin{array}{rl}
\tau^*_n\,=\, \displaystyle\max_{\A_g\succeq0} \:{\left\{\,\displaystyle\sum_{g\in G^{\star}_n}\log\det\, \A_g:\:
 \displaystyle\sum_{g\in G^{\star}_n}s_{n-d_g}\,=\,\
 \displaystyle\sum_{g\in G^{\star}_n}\v_{n-d_g}(\x)^T\A_g\v_{n-d_g}(\x)\,g(\x)\,,\quad\forall\x\in\R^d\,\right\}}\,.
 %\end{array}
\end{equation}
As proved in \cite[Theorem 6, p. 945]{lass-cras}, \eqref{def-new-dual} is a dual of \eqref{def-new-relax}
and also of \eqref{def-new-min}, that is,
weak duality $\tau^*_n\leq -\tau_n\leq-\rho_n$ holds, and in fact even strong duality holds\footnote{In the present context, the condition $\1\in Q_n(G)$ 
in \cite[Theorem 6]{lass-cras} is satisfied as soon as the constraint in 
\eqref{def-new-dual} is satisfied for some
matrices $\A_g\succ0$, for every $g\in G^\star_n$.} , i.e., $\tau^*_n=-\tau_n=-\rho_n$. So the convex relaxation \eqref{def-new-relax} is in fact exact (i.e. has same optimal value as \eqref{def-new-min}). 

\begin{theorem}
\label{th1}
 Let $n\in\N$ be fixed, arbitrary, and $S^{\star}=S^{O}$, $S^{\square}$, or $S^{\triangle}$ (fixed). Then
 
 (i)  The equilibrium measure  $\phi^{\star}$ of $S^{\star}$ is an optimal solution of \eqref{def-new}, the 
 optimal moment matrices  $\M_{n-d_g}(g\,\phi^{\star})$, $g\in G^\star_n$, are  unique, and the vector of moments $\bphi^\star=(\phi^\star_{\balpha})_{\balpha\in\N^d_{2n}}$ is the unique optimal solution of 
 \eqref{def-new-relax}.

 (ii) The unique optimal solution $(\A_g^\star)_{g\in G^\star_n}$  of problem \eqref{def-new-dual}
 satisfies $\A_g^\star=\M_{n-d_g}(g\,\phi^\star)^{-1}$, 
 %$\A_1^*=\M_{n-1}(g\phi^*)^{-1}$, 
 and therefore
  \begin{equation}
  \label{put-design-ball}
  \displaystyle\sum_{g\in\,G^\star_n}s_{n-g}\,=\,\displaystyle\sum_{g\in\,G^\star_n}
 \Lambda^{g\,\phi^\star}_{n-d_g}(\x)^{-1}\,g(\x)\,=\,\displaystyle\sum_{g\in\,G^\star_n}
 K^{g\,\phi^\star}_{n-d_g}(\x,\x)\,g(\x)\,,\quad\forall \x\in\R^d\,.
 \end{equation}

 (iii) There are $s_n\leq m\leq s_{2n}$ distinct points $\x^*_j\in S^\star$, and a positive vector 
 $0<\bgamma\in\R^m$ such that  the atomic probability measure $\nu^\star_n:=\sum_{j=1}^m \gamma_j\,\delta_{\x^*_j}$ 
 has the same moments as $\phi^\star$, up to degree $2n$. The normalized CD polynomial
 \begin{equation}
 \label{p-star-ball}
 \x\mapsto K^{\phi^\star}_n(\x,\x)
 \end{equation}
 attains its maximum $\sum_{g\in G^\star_n}s_{n-d_g}$  at all points of the boundary $\partial S^\star$ if $S^\star$ is the Euclidean unit ball, 
 only at the $2^d $ vertices %$\x=(\pm 1,\ldots,\pm 1)$ 
 of $\partial S^\star$,
 if $S^\star$ is the unit box, and 
 only at the $d+1$ vertices of $S^\star$%at all vertices of $S^\star$ 
 if $S^\star$ is the canonical simplex.

 (iv) With $\nu^\star_n$ as in (iii), the sequence $(\nu^\star_n)_{n\in\N}$ 
 converges to $\phi^\star$ for the weak-star topology of $\mathscr{M}(S^\star)$, i.e.,
 \[\lim_{n\to\infty}\int_S f\,d\nu^\star_n\,=\,\int_S f\,d\phi^\star\,,\quad\forall f\in\mathscr{C}(S^\star)\,.\]
 \end{theorem}
 
 \begin{proof}
 For (i) and (ii) see \cite{lass-cras} and
 \cite[Theorem 4.1, 4.3, 4.5]{lass-tams}. For the first statement in (iii) see
 e.g. \cite[Theorem 1]{dghl21}.
 %and  the many references therein.
 For the second statement in (iii), observe that by \eqref{put-design-ball}  one obtains
\[\displaystyle\sum_{g\in G^\star_n}s_{n-d_g}-K^{\phi^\star}_n(\x,\x)\,=\,
\sum_{1\neq g\in G^\star_n}\,K^{g\phi^\star}_{n-d_g}(\x,\x)\,g(\x)\quad\geq\,0\,,\quad\forall \x\in S^\star\,,\]
as $g\geq0$ on $S^\star$ for all $g\in G^\star_n$, and so $K^{\phi^\star}_n(\x,\x)\leq \sum_{g\in G^\star_n}s_{n-d_g}$ for all $\x\in S^\star$.
%Next, integrating w.r.t. $\phi^\star$ yields
%\begin{eqnarray*}\int_{S^\star}
%(\displaystyle\sum_{g\in G^\star_n}s_{n-d_g}-
%K^{g\phi^\star}_{n-d_g}(\x,\x)\,g(\x))\,d\phi^\star\\
%&=&\displaystyle\sum_{g\in G^\star_n}s_{n-d_g}-
%\left\langle\M_{n-d_g}(g\,\bphi^*)^{-1},\int_{\S^\star}\v_{n-d_g}(\x)\,\v_{n-d_g}(\x)^T\,g(\x)\,d\phi^\star\right\rangle\\
%&=&\displaystyle\sum_{g\in G^\star_n}s_{n-d_g}-
%\left\langle\M_{n-d_g}(g\,\bphi^*)^{-1},\M_{n-d_g}(g\,\bphi^*)\right\rangle\,=\,0\,,
%\end{eqnarray*}
%and  so 
%\[\displaystyle\sum_{g\in G^\star_n}s_{n-d_g}\,=\,\displaystyle\sum_{g\in G^\star_n}K^{g\,\phi^\star}_{n-d_g}(\x,x)\,g(\x)\,,\quad\phi^*{a.e.}\]

Next, recall that if $g=\1$ then $K^{g\,\phi^\star}_{n-d_g}(\x,\x)=K^{\phi^\star}_n(\x,\x)$. Hence
if $S^\star=S^O$, then $K^{\phi^\star}_n(\x,\x)=s_n+s_{n-1}$ 
for every $\x\in\partial S^O$ as $\Vert\x\Vert=1$ on $\partial S^O$. On the other hand, if $S^\star=S^\square$, observe that only at the $2^d$ points $\u=(\pm 1,\pm 1,\ldots,\pm 1)$
$g(\u)=0$ for every %vertex $\u$ of $[-1,1]^d$ and every 
$g\in G^\square_n$ with $g\neq1$.
Similarly, if $S^\star=S^\triangle$ then only at the vertices $\u$ of $S^\star$,
$g(\u)=0$ %only at the origin  $0\neq\u$ of $S^\triangle$ and every 
for every $g\in G^\triangle_n$ with $g\neq1$. This yields the desired conclusion.

Finally for (iv), as $S^\star$ is compact, there exists a subsequence $(n_k)_{k\in\N}$ and a probability measure $\psi$ on $S$ such that $\lim_{k\to\infty}\int_S f\,d\nu^\star_{n_k}=\int_{S^\star} f\,d\psi$ for all $f\in\mathscr{C}(S^\star)$. In particular, by definition of $\nu^\star_n$,
\[\phi^\star_{\balpha}\,=\,\lim_{k\to\infty}\int_{S^\star}\x^{\balpha}\,d\nu^\star_{n_k}\,=\,\int_{S^\star}\x^{\balpha}\,d\psi\,,\quad\forall\balpha\in\N^d\,,\]
and so as $S^\star$ is compact, $\psi=\phi^\star$. But this also implies that the whole sequence $(\nu^\star_n)_{n\in\N}$  converges to $\phi^\star$.
\end{proof}

%\begin{center}
\emph{
Hence when $S^\star$ is the unit ball, the unit box, or the simplex, its equilibrium measure $\phi^\star$ is an optimal solution of \eqref{def-new} for \emph{all} degrees $n$. Therefore the support of any available cubature for $\phi^\star$, with atoms in $S^\star$, positive weights, 
and exact up to degree $2n$, provides an optimal design.}
%\end{center}
}

Concerning  the construction of such atomic-measures in 
Theorem \ref{th1}(iii), the interested  reader is referred to \cite{cubature-ball} where several cubature formula are provided for the unit ball and unit sphere for various weight functions (including
the Chebyshev weight $1/\sqrt{1-\Vert\x\Vert^2}$). For numerical computation of cubatures the interested reader is referred to e.g. \cite{cools-1,cools-2,Gautschi}.

Incidentally, by Theorem \ref{th1}, with $s:=\sum_{g\in G^\star_n}s_{n-d_g}$,
the polynomials $(g(\x)\,K^{g\,\phi^\star}_{n-d_g}(\x,\x)/s)_{g\in G^\star_n}$ are nonnegative on $S^\star$,
and provide $S^\star$ with a partition of unity.

\begin{example}
Let $S^\star=S^\circ$ (the Euclidean unit ball). In the case 
$d=2$ (disk) and with $n=2$ (degree 4), 
let us compare the D-optimal designs obtained {with the approach in \cite{dghl21} briefly described in Section \ref{algo-description}, 
%{i.e., solving the convex relaxation \eqref{mom-conv-relax} with pseudo-moments, 
%associated with the original problem \eqref{def-D-optimal}, 
with the approach described in this paper, i.e., solving variant problem \eqref{D-variant}.}
\begin{figure}[h!]
	\begin{center}
		
		\includegraphics[width=\textwidth]{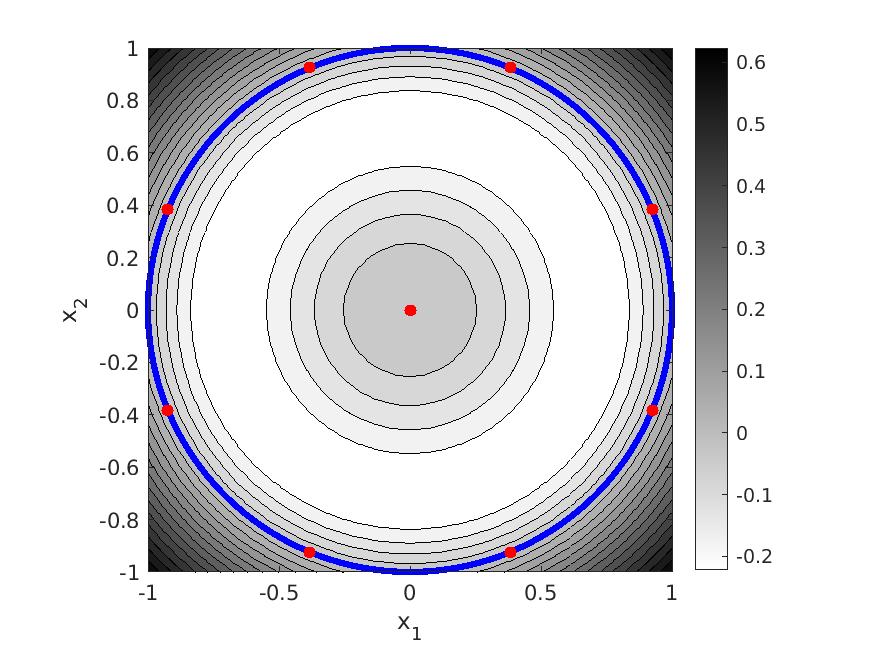}
		\caption{Degree 4 D-optimal design on the disk solving problem \eqref{def-D-optimal}. The optimal points (red) are located on the CD polynomial unit level set (blue).  \label{fig:ball2mom}}
	\end{center}
\end{figure}
\begin{figure}[h!]
	\begin{center}
		\includegraphics[width=\textwidth]{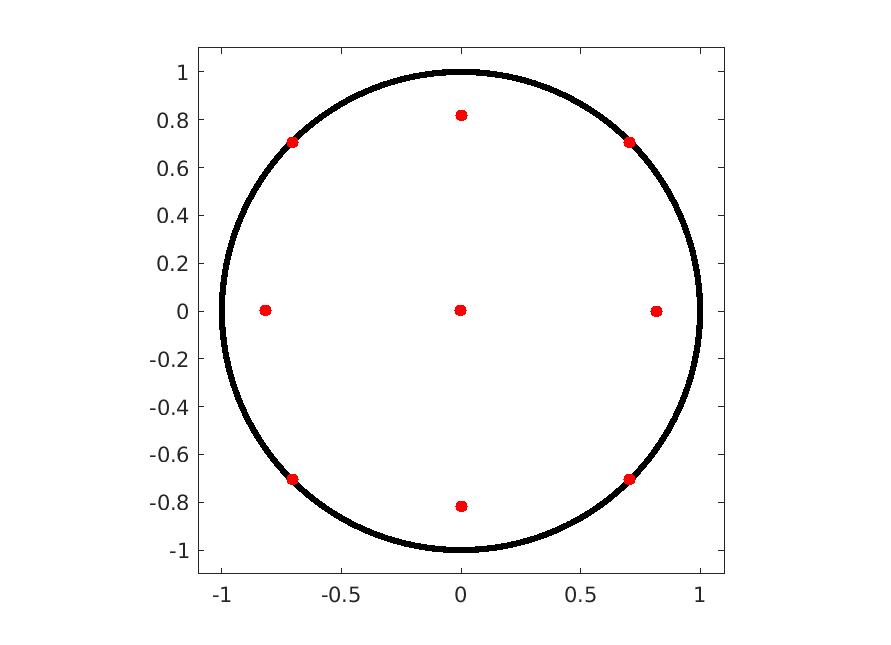}
		\caption{Degree 4 D-optimal design  on the disk (black) with points (red) solving variant problem \eqref{D-variant}.\label{fig:ball2momloc}}
	\end{center}
\end{figure}
{In solving \eqref{mom-conv-relax} at Step-1 of the two-step moment-SOS algorithm \cite{dghl21} (with $S$ described as $\{\x: 1-\Vert\x\Vert^2\geq0\}$ with one generator){,}}
a moment vector\footnote{The unique solution $\bmu^*_4$ is not guaranteed to come from a measure $\mu^*\in\mathscr{P}(S^\star)$, but in our numerical experiments in \cite{dghl21}, it does.} 
\[\bmu^*_4 := \int_{S^\star} \v_4(\x) d\mu(\x) \in \R^{15}\]
is computed numerically. Its non-zero entries are $\mu^*_{00}=1$, $\mu^*_{20}=\mu^*_{02}=0.4167$, $\mu^*_{40}=\mu^*_{04}=0.3125$, $\mu^*_{22}=0.1042$ (to 4 significant digits). In step-2, an atomic measure supported at 9 points is computed at relaxation order 5, i.e. by extending the moment vector up to degree 10. The points are displayed in red on Figure \ref{fig:ball2mom}. They are located on the unit level set (blue) of the CD polynomial. Other level sets (gray) are represented in logarithmic scale.

In solving {variant problem \eqref{D-variant}},
the moment vector $\bphi^\star_{4}\in\R^{15}$ of the equilibrium measure $\phi^\star$ of $S^\star$ 
is computed, either in closed form or numerically as above in step-1 (but with the new $\log\mathrm{det}$
criterion). Its non-zero entries read
$\phi^\star_{00}=1$, $\phi^\star_{20}=\phi^\star_{02}=1/3$, $\phi^\star_{40}=\phi^\star_{04}=1/5$, $\phi^\star_{22}=1/15$.
In step-2, an atomic measure supported on 9 points of $S^\star$ is computed at relaxation order 5, i.e. by extending the moment vector $\bphi^\star_4$ up to degree $10$. The points are displayed in Figure \ref{fig:ball2momloc}. As expected, we observe that 5 out of 9 points are in the interior of the disk, whereas only a single point (the origin) is in interior of the disk for original problem \eqref{def-D-optimal}.
\end{example}
\begin{example}
Let $S^*=S^\triangle$ (the simplex). With $d=2$ and $2n=4$, in solving the standard D-optimal design \eqref{def-D-optimal} via the two-step algorithm in \cite{dghl19} described in Section \ref{algo-description} {(with $S$ described as $\{\,\x: x_1\geq0,\,x_2\geq0\,,1-x_1-x_2\geq0\,\}$
with $3$ generators),}
one obtains 6 points (the three vertices and the mid-points on each facet). On the other hand in solving \eqref{D-variant} one obtains a cubature for the equilibrium measure, supported on $8$ points
on the boundary and one point inside.
\end{example}

\subsection*{The univariate case $S^\square=[-1,1]$}
Let $\phi^\star=dx/\pi\sqrt{1-x^2}$ and $x\mapsto g(x)=1-x^2$. Then
\eqref{put-design-ball} reads:
\begin{equation}
\label{aa1}
s_n+s_{n-1}\,=\,\Lambda^{\phi^\star}_n(x)^{-1}+(1-x^2)\,\Lambda^{g\cdot\phi^\star}_{n-1}(x)^{-1}\,,\quad\forall x\in\R\,,
\end{equation}
or, equivalently:
\begin{equation}
\label{aaa1}
s_n+s_{n-1}\,=\,K^{\phi^\star}_n(x,x)+(1-x^2)\,K^{g\cdot\phi^\star}_{n-1}(x,x)\,,\quad\forall x\in\R\,.\end{equation}
Next, the Gauss-Chebyshev quadrature for the equilibrium measure
$\phi^\star$ is supported on the zeros of the degree-$(n+1)$ Chebyshev polynomial of first kind, and reads
$\nu_n=\frac{1}{n+1}\sum_{i=1}^{n+1} \delta_{x_i}$ with
\begin{equation}
\label{gauss-chebyshev}
 x_i\,=\,\cos\left(\frac{2i-1}{2(n+1)}\pi\right)\in (-1,1)\,,\:i=1,\ldots,n+1\,,\end{equation}
and is exact up to degree $2n+2-1=2n+1$, i.e.,
\[\int_{-1}^1 p\,d\nu_n\,=\,\int_{-1}^1 p\,d\phi^*\,,\quad\forall p\in \R[x]_{2n+1}\,.\]
Moreover it turns out that 
\[K^{\phi^\star}_n(x_i,x_i)\,=\,\Lambda^{\phi^\star}_n(x_i)^{-1}\,=\,\Lambda^{\nu_n}_n(x_i)^{-1}\,=\,n+1=s_n\,,\]
and so by  \eqref{aa1},
$(1-x_i^2)\,\Lambda^{g\cdot\phi^\star}_{n-1}(x_i)^{-1}=s_{n-1}$, for all $i=1,\ldots,n+1$. 

So the distinguished equidistributed  atomic probability measure $\nu_n$  is an optimal solution of \eqref{def-new}. Moreover
\[\frac{\VDM(x_1,\ldots,x_{n+1})\VDM(x_1,\ldots,x_{n+1})^T}{n+1}
\,=\,\M_n(\nu_n)\,.\]
However, its support (zeros of the degree-$(n+1)$ Chebyshev polynomial) is \emph{not} a set of Fekete points because
\[\phi^\star\neq\arg\max_{\mu\in\mathscr{P}(S)}\mathrm{det}(\M_n(\mu))\,,\]
 (i.e., $\phi^\star$ does not solve the D-optimal design\footnote{An approximate D-optimal design for \eqref{def-D-optimal} is equally supported on the zeros of $(1-x^2)P'(x)$ where $P_n$ is the degree-$n$ Legendre polynomial; see \cite{hoel-2}. It is also a set of Fekete points \cite{Bos-2}. } \eqref{def-D-optimal}).
 Indeed $\max_{x\in S}K^{\phi^\star}_n(x,x)>s_n$ because by \eqref{aa1}
\[K^{\phi^\star}_n(x,x)\,\left(\,=s_n+s_{n-1}\right)\,>\,s_n\quad\mbox{ for $x=\pm 1$.}\]
So the points $(x_i)_{i\leq n+1}$ have the remarkable property that
the equidistributed measure $\nu_n$ (supported on $\mathrm{int}(S)$) 
is an optimal solution of \eqref{def-new} for every $n$.
And of course, again, $(\nu_n)_{n\in\N}$ converges weakly to the equilibrium measure
$\phi^\star$ as does any sequence of probability measures 
equi-distributed on Fekete points (see e.g. \cite{berman,book}).

However, even for $S$ being the unit ball or the unit sphere, cases $(d,n)$
 where a D-optimal  measure $\mu\in\mathscr{P}(S)$ 
 is equi-supported on Fekete points are exceptional; see e.g. Bos \cite[p. 134]{Bos-1}. The univariate atomic measure $\nu_n$ in \eqref{gauss-chebyshev} has an immediate extension to the multivariate $S=[-1,1]^d$ case.
 \begin{corollary}
 \label{coro-box}
 Let $\nu_n$ be the Gauss-Chebyshev atomic (univariate) measure in \eqref{gauss-chebyshev}.  Then the multivariate atomic product measure $\nu^\star_n:=\underbrace{\nu_n\otimes\cdots\otimes\nu_n}_{\mbox{$d$ times}}$ on $S=[-1,1]^d$ is an optimal solution of \eqref{def-new}.
 Moreover $\nu^\star_n$ converges to $\phi^\star$ for the weak-star topology of $\mathscr{M}(S^\star)$.
 \end{corollary}
 \begin{proof}  
 As $\phi^\star=\prod_{j=1}^d dx_j/(\pi(\sqrt{1-x_j^2})$ is an optimal
 solution of \eqref{def-new}, $\nu^\star_n$ provides with
 a cubature for $\phi^\star$, exact up to maximum degree $2n+1$ in 
 \emph{each} variable, and hence also exact for total degree $2n$.
{ As a result, $\nu^\star_n$ is an optimal solution of \eqref{def-new}.}
 The final statement on weak-star convergence is due to the product structure of $\phi^\star$ and $\nu^\star_n$.
 \end{proof}

 \begin{example}
In the case $d=1$ (interval) and $n=8$ (degree 16) let us compare the D-optimal designs obtained with the approach in \cite{dghl19} described in Section \ref{algo-description} for solving original problem \eqref{def-D-optimal}, and with the approach described in this paper for solving variant problem \eqref{D-variant}.

\begin{figure}[h!]
 \begin{center}
		\includegraphics[width=\textwidth]{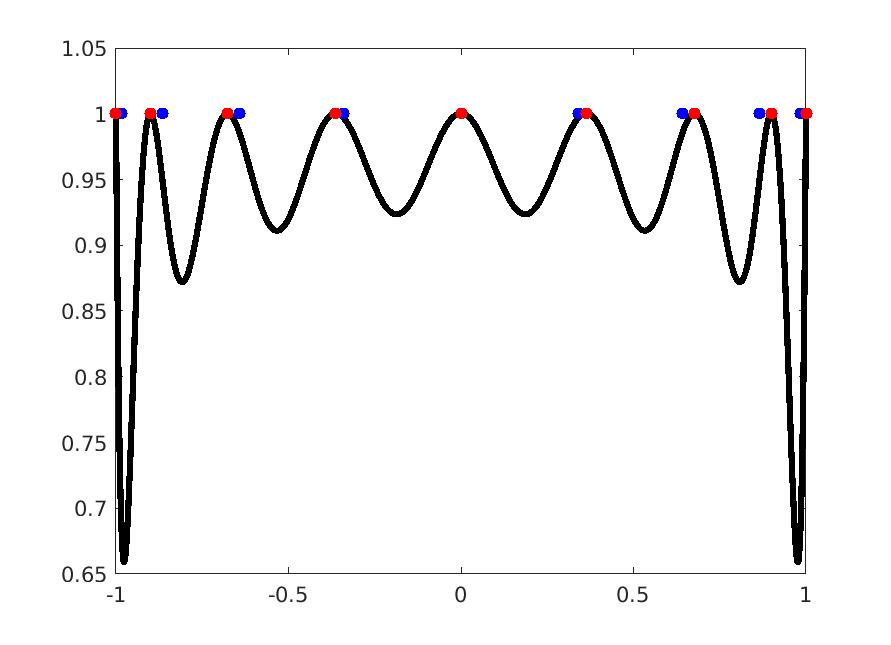}
		\caption{Degree 16 D-optimal design on unit interval: CD polynomial unit level set points (red) solve problem \eqref{def-D-optimal}, whereas roots of the degree 9 Chebyshev polnomials (blue) solve variant problem \eqref{D-variant}. \label{fig:int8}}
\end{center}
\end{figure}
In solving original problem \eqref{def-D-optimal}, in Step-1 of the moment-SOS algorithm one obtains a CD polynomial of degree $2n=16$ (black curve) whose local maxima (located on the unit level set) are the optimal 9 points (red dots), see Figure \ref{fig:int8}.
These points are also optimal Fekete points and roots of $x \mapsto (1-x^2)p'_n(x)$ where $p_n$ is the degree-$n$ Legendre polynomial \cite{Bos-2,hoel-2}. Also represented (blue dots) are the roots 
of the degree 9 Chebyshev polynomial, which cannot be optimal Fekete points, as discussed above.

In solving variant problem \eqref{D-variant}, as $\phi^\star=dx/(\pi\sqrt{1-x^2})$,
the roots of the degree 9 Chebyshev polynomial provide the support of an optimal
atomic probability measure. (Step-1 
of our moment-SOS algorithm \cite{dghl21} also provides the moment vector $\bphi^\star_{18}$ of $\phi^\star$,
rounded to $4$ significant digits.)
\end{example}

\subsection{Comparison with the original formulation}

An optimal atomic solution $\mu^*$ of \eqref{def-D-optimal} is such that
the degree-$2n$ SOS normalized CD polynomial $\x\mapsto 
K^{\mu^*}_n(\x,\x)/s_n$ satisfies
\[0\,\leq\,K^{\mu^*}_n(\x,\x)/s_n\,\leq 1,\quad\forall \x\in S\,,\]
and $\mu^*$ is supported on $r$ points $\z_i\in S$, $i=1,\ldots,r$, with $s_n\leq r\leq s_{2n}$, and $K^{\mu^*}_n(\z_i,\z_i)=s_n$ for all $i=1,\ldots,r$.
If $r=s_n$ then the points $\{z_1,\ldots,z_r\}$ form a set of {Fej\'er points and are also Fekete points}. However this case is not to be expected in general and even in special cases 
as the ones considered in this paper; see e.g. \cite{Bos-1,Bos-2}. Notice that an optimal measure $\mu^*\in\mathscr{P}(S)$ of \eqref{def-D-optimal} is not provided explicitly and in view of Theorem \ref{th1}(iii), $\mu^*$
cannot be the equilibrium measure $\phi^\star$  of $S$.
Only the optimal moment matrix
$\M_n(\mu^*)$ is obtained and is unique (while $\mu^*$ is not unique). 

On the other hand, if one considers the proposed variant
\eqref{def-new} of the D-optimal design problem, then remarkably for 
$S^\star$ being the unit ball, the unit box, or the simplex:

-- For every degree $n$ and every dimension $d$, the associated equilibrium measure $\phi^\star$ of $S^\star$ is an optimal solution.

-- Every cubature of $\phi^\star$ with positive weights, atoms in $S^\star$, and exact up to degree $2n$, provides with a D-optimal atomic measure $\nu^\star_n$ {whose support} is an optimal design.

-- The sequence $(\nu^\star_n)_{n\in\N}$ converges to the equilibrium measure $\phi^\star$ of $S^\star$
for the weak-star topology of $\mathscr{M}(S^\star)$.

\subsection{Computing an optimal atomic measure}
As we have seen, it reduces to that of searching for \emph{any} degree-$2n$ \emph{cubature} with positive weights
and atoms in $S^\star$, for the (known) equilibrium measure $\phi^\star$ of $S^\star$. For the unit box $[-1,1]^n$, the (tensorized) Gauss-Chebyshev cubature already provides a D-optimal design with no computation; see Corollary \ref{coro-box}. For the Euclidean ball and the simplex, some cubatures 
for $\phi^\star$ are also known in some cases. For numerical procedures for computing cubatures, the interested reader is referred to \cite{cools-1,cools-2,Gautschi}.
{
An alternative is to adapt the two-step procedure described in \cite{dghl19} and sketched in Section \ref{algo-description}, to solve the variant \eqref{def-new}.
}

\section{More general semi-algebraic sets}

For the special geometries $S^\star$ of Section \ref{geometry}, we have seen that the infinite-dimensional problem \eqref{def-new} is also equivalent to solving the single finite-dimensional convex relaxation \eqref{def-new-relax}; see Theorem \ref{th1}.
%\begin{equation}\label{D-new-mom}
%  \displaystyle\max_{\bphi=(\phi_{\balpha})\in\R^{s_{2n}}} %\,\{\,\displaystyle\sum_{g\in G^\star_n}
 % \log\det\, \M_{n-d_g}(g\,\bphi):\:\phi_{0}\,=\,1\,;\:
 % \M_{n-d_g}(g\,\bphi)\,\succeq0\,\: g\in G^\star_n\,\}\,,
 %\end{equation}
% where $\M_{n-d_g}(g\cdot\bphi)$) is the localizing matrix associated with the sequence
 %$\bphi=(\phi_{\balpha})_{\balpha\in\N^d_{2n}}$ of \emph{pseudo-moments} and the polynomials $g\in G^\star_n$. Notice that
 %\eqref{D-new-mom} is a relaxation of \eqref{def-new} as one now considers
 %a vector $\bphi$ of pseudo-moments (up to degree $2n$) rather than the vector of moments 
 %up to degree $2n$ of a measure $\phi$ on $S^\star$.  
%However and remarkably, by \cite{lass-tams}, Problem \eqref{D-new-mom}
 %has a unique optimal solution, the sequence $\bphi^\star_{2n}$
 %of moments (up to degree $2n$) of the equilibrium measure $\phi^\star$ of $S^\star$
 %(uniqueness is due to strict concavity of the criterion). Namely:
 
%\begin{theorem}[\cite{lass-tams}[Theorem 1]]
% \label{th2}
%With $n\in\N$ fixed, the semidefinite program \eqref{D-new-mom} has a unique optimal solution $\bphi^\star_{2n}=(\phi^\star_{\balpha})\in\R^{s_{2n}}$ which satisfies 
%\begin{equation}
 %\phi^\star_{\balpha}\,=\,\int_{S^\star} \x^{\balpha}\,d\phi^\star\,,\quad\forall\balpha\in\N^d_{2n}\,,
%\end{equation}
%where $\phi^\star$ is the equilibrium measure of $S^\star$ (optimal solution of \eqref{def-new}).
%\end{theorem}
So while the equilibrium measure $\phi^\star$ is not the unique optimal solution of  \eqref{def-new} (any probability measure on $\R^d$ with same degree-$2n$ moments as $\phi^\star$ is also optimal), its vector $\bphi^\star_{2n}$ of moments up to degree $2n$, is the unique optimal solution of \eqref{def-new-relax}.

\paragraph{Basic semi-algebraic sets $S$} Of course, for the special geometries $S^\star$ (unit ball, unit box or simplex), we know $\phi^\star$ (and so $\phi^\star_{2n}$ as well) and
there is no need to solve \eqref{def-new-relax}. But for
other geometries, e.g., if $S\subset\R^d$ is the compact basic semi-algebraic set
\begin{equation}
    \label{set-S}
S\,=\,\{\,\x\in\R^d:\: g_j(\x)\geq0\,,\:j=1,\ldots,m\,\}\,,\,\end{equation}
for some $g_j\in\R[\x]$, %of degree $d_j$, 
$j=1,\ldots,m$, then inspired by \eqref{def-new}, we propose the following
analogue of the D-optimal variant \eqref{D-variant}:%, \eqref{def-new-box}, \eqref{def-new-simplex}:
\begin{equation}
\label{D-general-mom}
  \displaystyle\max_{\phi\in \mathscr{P}(S)} \:\sum_{j=0}^m\log\det\, \M_{n-r_j}(g_j\cdot\bphi)\,,
  \end{equation}
 where $g_0=\1$, and $r_j=\lceil\mathrm{deg}(g_j)/2\rceil$, $j=0,\ldots,m$. The set
 \begin{equation}
 \label{quad-module}
 Q_n(g_1,\ldots,g_m)\,:=\,\{\,\x\mapsto \sum_{j=0}^m \sigma_j(\x)\,g_j(\x)\,:\quad \sigma_j\in\Sigma[\x]\,;\: \mathrm{deg}(\sigma_j\,g_j)\,\leq\,2n\,\}\,\end{equation}
 is called the truncated quadratic module associated with the polynomials $g_1,\ldots,g_m$,
 a central object in the Moment-SOS hierarchy for polynomial optimization; see e.g.
\cite{lass-book,lass-acta}. As  $S$ is compact, $S\subset\{\,\x: M-\Vert\x\Vert^2\geq0\,\}$ for some $M>0$,
and therefore, we may and will also assume that $g_1(\x)=M-\Vert\x\Vert^2$ (if not we include 
this redundant quadratic constraint in the definition of $S$ without changing $S$).

 \begin{theorem}
 \label{th-S}
  Let $S\subset\R^d$ in \eqref{set-S} be compact with nonempty interior.
  Then
  every optimal solution $\mu\in\mathscr{P}(S)$ of \eqref{D-general-mom} satisfies
  \begin{eqnarray}
      \label{th-S-1}\sum_{j=0}^m s_{n-r_j}-\sum_{j=0}^mg_j(\x)\,K^{g_j\cdot\mu}_{n-r_j}(\x,\x)&\geq&0\,,
      \quad\forall \x\in S\\
      \label{th-S-2}\sum_{j=0}^ms_{n-r_j}-\sum_{j=0}^m g_j(\x)\,K^{g_j\cdot\mu}_{n-r_j}(\x,\x)&=&0\,,\quad\mbox{$\mu$-a.e.}
  \end{eqnarray}
In addition, if $\mu=f\,d\x$ with $f>0$ on $S$, is an optimal solution then
\begin{equation}
\label{th-S-3}
    \sum_{j=0}^m s_{n-r_j}\,=\,\sum_{j=0}^m g_j(\x)\,K^{g_j\cdot\mu}_{n-r_j}(\x,\x)\,,\quad\forall \x\in\R^d.
\end{equation}
 \end{theorem}

\begin{proof}
  Introduce the two (finite-dimensional) convex cones
\begin{eqnarray*}
      \mathcal{C}_n&:=&\{\,\bphi=(\phi_{\balpha})_{\balpha\in\N^d_{2n}}:\:(\phi_{\balpha})\,=\,\left(\int_S\x^{\balpha}\,d\phi\right)\mbox{ for some $\phi\in\mathscr{M}(S)_+$}\}\\
      \mathcal{P}_n&:=&\{\,p\in\R[\x]_{2n}: \:\mbox{$p\geq0$ on $S$}\}\,,
      \end{eqnarray*}
      and let $\mathcal{C}^*_n$ and $\mathcal{P}^*_n$ denote their respective duals. {
      From the Riesz-Haviland Theorem (see e.g. \cite[Theorem 2.34]{lass-book}) and since $S$ is compact with nonempty interior,} it turns out that $\mathcal{C}_n=\mathcal{P}^*_n$ and $\mathcal{C}_n^*=\mathcal{P}_n$.
      For every $j\leq m$, let the real symmetric matrices $\B^j_{\balpha}$ {of size $n-r_j$} be defined by
      \[\sum_{j=0}^m g_j(\x)\,\v_{n-r_j}(\x)\v_{n-r_j}(\x)^T\,=\,\sum_{\balpha\in\N^d_{2n}}\B^j_{\balpha}\,\x^{\balpha}\,,\:\forall \x\in\R^d\,.\]
      {Note that by construction it holds $$\M_{n-r_j}(g_j\,\bphi) = \sum_{\balpha \in \N^d_{2n}} \B^j_{\balpha} \phi_{\balpha}.$$}
      Then \eqref{D-general-mom} also reads:
   \begin{equation}
\label{D-general-mom2}
 \displaystyle\max_{\bphi\in \mathcal{C}_n} \:\{\,\sum_{j=0}^m\log\det\, \M_{n-r_j}(g_j\,\bphi)\,:\:\phi_0=1\,\}\,,
  \end{equation}   
  which is also a convex program. 
  Let $\mu^*$ be an optimal solution of 
  \eqref{D-general-mom} and let $\bmu^*_{2n}\in\mathcal{C}_n$ be its vector of
  moments up to degree $2n$ (hence an optimal solution of \eqref{D-general-mom2}). 
  {Recalling that ${\mathbf X}^{-1}$ is the gradient of $\log\det\mathbf X$}, the necessary KKT-optimality conditions  
  imply that there exists some scalar $\lambda^*$ and some polynomial $q^*\in\mathcal{P}_n$ such that
  \[\lambda^*1_{\balpha=0}-\sum_{j=0}^m \langle\, \M_{n-r_j}(g_j\,\bmu^*_{2n})^{-1},\B^j_{\balpha}\,\rangle\,=\,q^*_{\balpha}\,,\:\forall\balpha\in\N^d_{2n}\,;\quad \langle \bmu^*_{2n},q^*\rangle\,=\,0\quad\mbox{[complementarity]}\,.\]
Multiplying by $(\mu^*_{2n})_{\balpha}$ and summing up, yields $\lambda^*=\sum_{j=0}^ms_{n-r_j}$,
and multiplying by $\x^{\balpha}$ and summing up yields the desired result \eqref{th-S-1}-\eqref{th-S-2}. To get \eqref{th-S-3} observe that  with $f>0$ on $S$, and $q^*\in \mathcal{P}_n$
\[0\,=\,\langle\bmu^*_{2n},q^*\rangle\,=\,\int_S q^*\,d\mu^*\,=\,\int_S q^*\,f\,d\x\quad\Rightarrow\quad \mbox{$q^*=0$,  a.e. on $S$}\,,\]
and since $S$ has nonempty interior, $q^*=0$ for all $\x\in\R^d$.
\end{proof}
{Of course \eqref{D-general-mom} (or, equivalently \eqref{D-general-mom2}) is not solvable in general because no tractable description of $\mathcal{C}_n$ is available. Therefore} we also introduce
its convex relaxation (the analogue of \eqref{mom-conv-relax})
\begin{equation}
\label{general-mom-relax}
 \displaystyle\max_{\bmu=(\mu_{\balpha})} \left\{ \displaystyle\sum_{j=0}^m\log\det \,\M_{n-r_j}(g_j\,\bmu):\:
 \mu_0=1\,;\quad\M_{n-r_j}(g_j\,\bmu)\,\succeq\,0\,,\:j=0,\ldots,m\,\right\}\,,
 \end{equation}
where $\bmu=(\mu_{\balpha})_{\balpha\in\N^d_{2n}}$ is a vector of pseudo-moments up to degree $2n$.
 
In Step-1 of the algorithm described in \cite{dghl19,dghl21} and in Section \ref{algo-description} to solve \eqref{def-D-optimal}, one solves almost the same problem as \eqref{general-mom-relax} except that the 
criterion is simply $\log\,\mathrm{det}\,\M_n(\bmu)$. %{(e.g. see (14) in \cite{dghl19})}. 
So it is straightforward to adapt its Step-1 for solving \eqref{general-mom-relax}. Then Step-2 (whose input is the 
output $\bmu^*_{2n}$ of Step-1), remains exactly the same. Its goal is to extract the support of 
an atomic probability measure on $S$ with same moments as $\bmu^*_{2n}$, that is, an optimal design.
Equivalently, \eqref{general-mom-relax} reduces to
\begin{equation}
\label{D-general-mom-inf}
 \rho_n\,=\,\displaystyle\min_{\bmu=(\mu_{\balpha})}\,\left\{\,-\displaystyle\sum_{j=0}^m\log\det\, \M_{n-r_j}(g_j\,\bmu):\:
 {\mu_0}=1\,;\quad
 \M_{n-r_j}(g_j\, \bmu)\,\succeq\,0\,,\:\forall j\leq m\,\right\}\,,
  \end{equation}
whose optimal value $\rho_n$ is minus that of \eqref{general-mom-relax}. Consider the convex optimization problem:
\begin{equation}
\label{D-general-mom-dual}
 \rho^*_n\,=\,\displaystyle\max_{\A_j\succeq0}\,\left\{\,\displaystyle\sum_{j=0}^m\log\det\,\A_j:\:
  \displaystyle\sum_{j=0}^ms_{n-r_j}\,=\,\displaystyle\sum_{j=0}^m g_j(\x)\,\v_{n-r_j}(\x)^T\A_j\,\v_{n-r_j}(\x)\,,\quad\forall \x\in\R^d\,\right\}\,.
 \end{equation}
 As  shown in \cite{lass-cras} and \cite[Theorem 3.2]{lass-tams}, if $\1\in\mathrm{int}(Q_n(g_1,\ldots,g_m))$\footnote{If $g_1(\x)=M-\Vert\x\Vert^2$ then the condition $\1\in\mathrm{int}(Q_n(g_1,\ldots,g_m))$ is satisfied; see \cite{lass-cras}.}
 then \eqref{D-general-mom-dual} is a dual of \eqref{D-general-mom-inf}, and strong duality holds, i.e., $\rho_n=\rho^*_n$.  
Moreover 
\eqref{D-general-mom-inf} (resp. \eqref{D-general-mom-dual}) has a unique optimal solution 
$\bmu^*_{2n}$ (resp. $(\A^*_j)$), and $\A^*_j=\M_{n-r_j}(g_j\cdot\bmu^*_{2n})^{-1}$ for all $j=0,\ldots,m$. Therefore
\begin{equation}
\label{aux-mu}
\displaystyle\sum_{j=0}^ms_{n-r_j}=\displaystyle\sum_{j=0}^m g_j(\x)\,K^{g_j\cdot\bmu^*_{2n}}_{n-r_j}(\x,\x)\,,\quad\forall\x\in\R^d\,.
\end{equation}

{The difference with the three special cases of $S$ in Section \ref{geometry}, is that now the linear functional $\bmu^*_{2n}\in\R[\x]_{2n}^*$ is not guaranteed to have a representing measure $\mu^*$ on $S$ (Step-1 of the algorithm in \cite{dghl21} and  its adaptation {\eqref{general-mom-relax}} proposed for \eqref{D-variant} assumes that it is the case), let alone $\mu^*$ being the equilibrium measure of $S$. However as shown in {\cite[Theorem 10]{lass-cras}}, each accumulation point of the sequence $(\bmu^*_{2n})_{n\in\N}$ has a representing probability measure $\mu^*$ on $S$.}

\begin{remark}
Recall that if \eqref{D-general-mom} has an optimal solution $\mu\in\mathscr{P}(S)$ such that $\mu=fd\x$ with
$f>0$ on $S$, and if $S$ has nonempty interior, 
then by Theorem \ref{th-S}
\begin{equation}
\label{aux-mu-1}
\sum_{j=0}^m s_{n-r_j}=\sum_{j=0}^m g_j(\x)\,K^{g_j\cdot\mu}_{n-r_j}(\x,\x)\,,\quad\forall \x\in\R^d\,,\end{equation}
and its vector $\bmu_{2n}:=(\mu_{\balpha})_{\balpha\in\N^d_{2n}}$ of moments up to degree $2n$
is an optimal solution of \eqref{D-general-mom2} of which \eqref{general-mom-relax} is a relaxation. {But then
from \eqref{aux-mu-1} and \eqref{aux-mu}, it follows that 
$(\A_j:=\M_{n-r_j}(g_j\,\mu)^{-1})_{j=0,\ldots,m}$ form an optimal solution of \eqref{D-general-mom-dual}
and $\bmu_{2n}:=(\mu_{\balpha})_{\balpha\in\N^d_{2n}}$ is the unique optimal solution of 
\eqref{general-mom-relax}.}
In other words, the relaxation \eqref{general-mom-relax} (or \eqref{D-general-mom-inf}) is \emph{exact}. That is, the vector of pseudo moments $\bmu^*_{2n}$, unique optimal solution of \eqref{general-mom-relax}, 
is in fact the vector of moments (up to degree $2n$) of $\mu$.
\end{remark}

\subsection*{Link between $\bmu^*_{2n}$ and the equilibrium measure of $S$}
{Even if an optimal solution $\bmu^*_{2n}$ of \eqref{general-mom-relax} (or, equivalently \eqref{D-general-mom-inf}) 
has no representing measure, can we say something on possible links to the equilibrium measure $\phi^*$ of $S$?

With $\phi^*$ being the equilibrium measure of $S$, introduce the polynomial
$p^*_n\in\R[\x]_{2n}$ defined by:
\begin{equation}
\label{aux-phi*-1}
\x\mapsto (\sum_{j=0}^m s_{n-r_j})\,p^*_n(\x)\,:=\,\sum_{j=0}^m g_j(\x)\,K^{g_j\cdot\phi^*}_{n-r_j}(\x,\x)\,,\quad\forall\x\in\R^d\,,
\end{equation}
If $p^*_n$ would be the constant polynomial equal to $1$, then in view of \eqref{aux-mu}, 
the moment sequence $\bphi^*_{2n}$ of $\phi^*$ 
would be an optimal solution of \eqref{general-mom-relax}.}  

{Under some regularity conditions on $(S,\phi^*)$:
\begin{itemize}
    \item The sequence of measures $(p^*_n\phi^*)_{n\in\N}$ weak star converges to $\phi^*$,
    \item and under an additional condition, one obtains the stronger result, $\lim_{n\to\infty}p^*_n(\x)=1$, uniformly on compact subsets of $\mathrm{int}(S)$,
    which shows that somehow \eqref{aux-phi*-1} is close to \eqref{aux-mu} (at least when $\x\in\mathrm{int}(S)$).
\end{itemize}
Indeed fix $j\leq m$ arbitrary, and consider the measure $\nu_j:=g_j\phi^*$ on $S$ and its associated Christoffel polynomial $K^{\nu_j}_n(\x,\x)$. We also assume that
$g_j>0$ on $\mathrm{int}(S)$.
If $(S,\nu_j)$ satisfies the Bernstein-Markov property (see \cite[(1.3) p. 603]{kroo}) then
by \cite[Theorem 4.4.4]{book},
\[\lim_{n\to\infty}\int_S f(\x) \frac{K^{g_j\,\phi^*}_n(\x,\x)}{s_n}\,g_j(\x)\,d\phi^*\,=\,\lim_{n\to\infty}\int_S f(\x) \frac{K^{\nu_j}_n(\x,\x{)}}{s_n}\,d\nu_j\,=\,\int_S f\,d\phi^*\,,\quad \forall f\in\mathscr{C}(S)\,.\]
Hence as $j$ was arbitrary, it follows that
\begin{equation}
    \label{weak-strong}
\lim_{n\to\infty}\int_S f(\x)\,p^*_n(\x)\,d\phi^*(\x)\,=\,\int_S f\,d\phi^*\,,\quad\forall f\in\mathscr{C}(S)\,,
\end{equation}
i.e., $p^*_n\,\phi^*\,\Rightarrow \phi^*$ as $n\to\infty$ (for the weak convergence of probability measures).

In addition, if $\lim_{n\to\infty} K^{\phi^*}_n(\x,\x)/s_n=1$, uniformly on compact subsets of $\mathrm{int}(S)$,
then
\[\lim_{n\to\infty}K^{\nu_j}_n(\x,\x)/s_n\,=\,1/g_j(\x)\,,\quad\mbox{uniformly on compact subsets of $\mathrm{int}(S)$}\,.\]
%where $f^*$ is the density of $\phi^*$ w.r.t. Lebesgue measure.
See e.g. \cite[Theorem 4.4.1]{book} (which is a consequence
of (1.5) in \cite[Theorem 1.1]{kroo} and \cite[Remark(b), p.63]{kroo}). Hence, as $j$ was arbitrary,
\[\lim_{n\to\infty}g_j(\x)\,K^{\nu_j}_n(\x,\x)/s_n\,=\,1\,\quad\mbox{uniformly on compact subsets of $\mathrm{int}(S)$,}\]
for every $0\leq j\leq m$, which in turn implies
\begin{equation}
\label{aux-phi*-2}\lim_{n\to\infty}p^*_n(\x)\,=\,1\,,
\quad\mbox{uniformly on compact subsets of $\mathrm{int}(S)$.}
\end{equation}}

Comparing \eqref{aux-mu} with \eqref{aux-phi*-1}-\eqref{aux-phi*-2} {shows} 
that even if the linear functional $\bmu^*_{2n}$, unique optimal solution of \eqref{general-mom-relax},  does not have a representing measure, asymptotically it is strongly connected to the equilibrium measure $\phi^*$ of $S$. Indeed, for sufficiently large degree $n$ and for every $\x\in\mathrm{int}(S)$:
\[\sum_{j=0}^m g_j(\x)\,K^{g_j\cdot\phi^*}_{n-r_j}(\x,\x)\,\approx\,\sum_{j=0}^m s_{n-r_j}\,=\,\displaystyle\sum_{j=0}^m g_j(\x)\,K^{g_j\cdot\bmu^*_{2n}}_{n-r_j}(\x,\x)\,.\]
Hence this asymptotic property supports our claim that the variant \eqref{D-general-mom} of the D-optimal design problem 
still  has a strong connection with the equilibrium measure of $S$.

\section{Conclusion}

We have introduced a variant of the D-optimal 
design problem with a more general information matrix which takes into account the geometry of the design space $S\subset\R^d$. One main reason to introduce such a variant is that remarkably, for the three cases where $S$ is the Euclidean ball, the  unit box and the canonical simplex (in any dimension),
the equilibrium measure $\phi^*$ of $S$ (in pluripotential theory) provides an optimal solution for every degree $n$. Therefore finding a D-optimal design 
(for this variant) reduces to finding a cubature for $\phi^*$, with positive weights, atoms in $S$, and exact up to degree $2n$.
Moreover and trivially, the associated resulting sequence of atomic probability measures $(\nu^*_n)_{n\in\N}$ converges to $\phi^*$ for the weak-star topology. 

If the link between 
statistics and approximation theory has been largely invoked in the literature since the pioneer 
works \cite{kief-wolf} and \cite{karlin}, this new variant 
makes it even stronger and striking, at least for the three specific cases of $S$.
Hence such a remarkable property suggests that one might
use the log-det criterion of this new information matrix, even for more general compact basic semi-algebraic sets
$S\subset\R^d$. In particular, the two-step algorithm proposed in \cite{dghl21} is easily adapted to this new variant of the D-optimal design problem.

\end{document}